\documentclass[twoside]{article}
\usepackage{latexsym}
\usepackage{pstricks}
\usepackage{amssymb}
\usepackage{amsmath}
\usepackage{amsfonts}
\usepackage{amsthm,array}
\usepackage{array}
\usepackage{epsfig}
\usepackage{graphicx}
\usepackage[numbers,sort&compress]{natbib}

\newtheorem{theorem}{Theorem}[section]
\newtheorem{proposition}{Proposition}[section]
\newtheorem{lemma}{Lemma}[section]
\newtheorem{corollary}{Corollary}[section]

\newtheorem{definition}{Definition}[section]

\newtheorem{remark}{Remark}[section]
\newtheorem{example}{Example}[section]

\newtheorem{remark-definition}{Remark and Definition}[section]
\newtheorem{rem-not}{Remark and Notation}[section]

\begin{document}
\title{\bf Verifiable Criteria and Properties for Interval $B_{\pi}^{R^I}$-Tensors\thanks{The first author's work was supported by the National Natural Science Foundation of P.R. China (Grant No.12171064).}
}
\date{}
\author{Li Ye, Yisheng Song\thanks{Corresponding author E-mail: yisheng.song@cqnu.edu.cn}\\
	School of Mathematical Sciences,  Chongqing Normal University, \\
	Chongqing, 401331, P.R. China. \\ Email: neutrino1998@126.com (Ye); yisheng.song@cqnu.edu.cn (Song)}
\maketitle

{\noindent\bf Abstract.} This paper introduces interval $B_{\pi}^{R^I}$-tensors as a natural extension of $B_{\pi}^{R}$-tensors to the interval setting. We provide two practical verifiable criteria for an interval tensor to be an interval $B_{\pi}^{R^I}$-tensor, one based on endpoint inequalities and another constructing an explicit vector $\pi$. Connections with interval $P$-tensors, positive definite interval tensors, and interval $Z$-tensors are established. Applications in polynomial optimization and interval tensor complementarity problems are briefly discussed.

\vspace{.3cm}
{\noindent\bf Mathematics Subject Classification.} 90C23, 15A69, 15A72, 90C30, 65G40
\vspace{.3cm}

{\noindent\bf Keywords.} $B_{\pi}^{R^I}$-tensor, Interval tensor, $P$-tensor, Positive definiteness.

\section{Introduction}
\setcounter{section}{1}
Tensors, as higher-order generalizations of vectors and matrices, are widely used in polynomial optimization, spectral hypergraph theory, magnetic resonance imaging, higher-order Markov chains, and tensor complementarity problems \cite{b1,b2,b3,b4,ql2017}. With the growing scale of practical applications, the study of algebraic structures and numerical properties of tensors has become an active research frontier. Among the various directions, the identification of structured tensors and their use in determining positive definiteness, providing eigenvalue bounds, and ensuring solvability of complementarity problems has received particular attention \cite{sq2015,dlq2018}.

In matrix theory, the class of $B$-matrices which is a subclass of $P$-matrices, has been extensively studied due to its favourable properties in providing error bounds for linear complementarity problems \cite{op2021}. Moreover, the verifiable conditions in its definition greatly facilitate algorithmic analysis. Inspired by this success, Song and Qi \cite{sq2015} systematically extended the concepts of $B$-matrices and $P$-matrices to higher-order tensors, introducing $B$-tensors and $P$-tensors. They established that an even-order symmetric $B$-tensor is positive definite. Subsequently, many results on $B$-tensors emerged. For example, Qi and Song \cite{qs2014} proved that every even-order symmetric $B$-tensor is positive definite. Li et al. \cite{cy2015} introduced double $B$-tensors and quasi-double $B$-tensors, enriching the structural decomposition theory of $B$-tensor classes. Li et al. \cite{d2} investigated the properties of MB-tensors and MB$_0$-tensors.

In 2019, Orera and Pe\~{n}a \cite{bpi2} introduced the more general class of $B_{\pi}^{R}$-tensors. Unlike classical $B$-tensors where all off-diagonal entries are bounded by the fixed constant $\frac{1}{n}$, a $B_{\pi}^{R}$-tensor employs a tunable vector $\pi$ and row sum based scaling $R_i$, which significantly enhances flexibility. They proved that every odd-order $B_{\pi}^{R}$-tensor is a $P$-tensor, and every symmetric even-order $B_{\pi}^{R}$-tensor is also a $P$-tensor, hence positive definite. Later, building on this framework, He et al. \cite{bpi1} studied several classes of nonsingular tensors and discussed their applications.

In many practical problems, model coefficients are not known exactly but only known to lie within given intervals. Interval analysis has been successfully applied to matrices and linear systems. In recent years, this idea was extended to tensors. Bozorgmanesh et al. \cite{boz2020} introduced interval tensors, discussing eigenvalue intervals, positive definiteness conditions, and applications in multilinear systems. In the same year, Rahmati and Tawhid \cite{sr2020} studied properties of intervals and sets of hypermatrices. Subsequently, Beheshti et al. \cite{bf2022} investigated properties of several classes of interval tensors, enriching the theoretical system. Cui et al. \cite{CZ2023} derived bounds for the H-eigenvalues of even order real symmetric interval tensors as well as negative interval tensors.

In the interval matrix setting, Lorenc \cite{lm2023} systematically characterized interval $B_{\pi}^{R^I}$-matrices, interval $B$-matrices and interval double $B$-matrices. Very recently, Ye and Song \cite{ys2026} studied interval $B$-tensors and interval double $B$-tensors, providing criteria based solely on the endpoint tensors. However, despite the rich theory of interval $B_{\pi}^{R^I}$-matrices, the corresponding study of interval $B_{\pi}^{R^I}$-tensors, which is the natural interval extension of $B_{\pi}^{R}$-tensors, remains unexplored.

This paper aims to fill this gap. We first provide a corrected characterization of $B_{\pi}^{R}$-tensors. We then define interval $B_{\pi}^{R^{I}}$-tensors and establish two practical verifiable criteria. Moreover, we connect interval $B_{\pi}^{R^{I}}$-tensors with interval $P$-tensors, positive definite interval tensors, and interval $Z$-tensors. Applications in polynomial optimization and interval tensor complementarity problems are discussed.

The paper is organized as follows. Section 2 recalls preliminaries. Section 3 studies properties of $B_{\pi}^{R}$-tensors. Section 4 develops the interval theory. Section 5 presents applications.

\section{Preliminaries}
\setcounter{section}{2}

Let $m$ and $n$ be positive integers. A real $m$th-order $n$-dimensional tensor $\mathcal{A}=(a_{i_1i_2\cdots i_m})$ is a multi-indexed array over the real field, where each index $i_j$ takes values in the index set $[n] := \{1,2,\dots,n\}$ for $j=1,2,\dots,m$. The space consisting of all such tensors is denoted by $T_{m,n}$. An entry $a_{i_1i_2\cdots i_m}$ is called a diagonal entry if $i_1=i_2=\cdots=i_m$, otherwise, it is called an off-diagonal entry. For any fixed index $i\in[n]$, the family of entries $a_{ii_2\cdots i_m}$ with $i_2,\dots,i_m\in[n]$ is referred to as the $i$th row of $\mathcal{A}$. The sum of all entries in the $i$th row, termed the $i$th row sum of $\mathcal{A}$, is defined as
$$R_i(\mathcal{A}):=\sum\limits_{i_2,\cdots,i_m=1}^na_{ii_2\cdots i_m}.$$
We write $R(\mathcal{A}):=(R_1(\mathcal{A}),R_2(\mathcal{A}),\cdots,R_n(\mathcal{A}))$ for the vector of row sums.
The tensor $\mathcal{I}\in T_{m,n}$ denotes the identity tensor whose diagonal entries equal to $1$ and all off-diagonal entries equal to $0$. A tensor $\mathcal{A}\in T_{m,n}$ is called symmetric if its entries are invariant under any permutation of their indices. For brevity, we use $[n]^{m-1}$ to denote the set of $(m-1)$-dimensional index tuples $(i_2,\cdots,i_m)$ with $i_2,\cdots,i_m\in[n]$, i.e., $[n]^{m-1}:=\{(i_2,\cdots,i_m):i_2,\cdots,i_m\in[n]\}$.

For a tensor $\mathcal{A}= (a_{i_1\cdots i_m})\in T_{m,n}$, and an $n\times n$ matrix $M=(m_{ij})$, the product $\mathcal{A}\times_1 M\times_2M\times_3\cdots\times_m M$ is an $m$th-order $n$-dimensional tensor whose entries are given by
$$(\mathcal{A}\times_1 M\times_2M\times_3\cdots\times_m M)_{i_1i_2\cdots i_m}=\sum\limits_{j_1j_2\cdots j_m=1}^n a_{j_1j_2\cdots j_m}m_{i_1j_1}\cdots m_{i_mj_m}.$$
For a vector $\mathbf{x}\in\mathbb{R}^n$, the product $\mathcal{A}\mathbf{x}^{m-1}$ is a vector in $\mathbb{R}^n$ whose $i$th component is defined by
$$(\mathcal{A}\mathbf{x}^{m-1})_i:=\sum\limits_{i_2,\cdots i_m=1}^n a_{ii_2\cdots i_m}x_{i_2}\cdots x_{i_m},\qquad 1\leq i \leq n.$$


\begin{definition}\label{def-bt}\textup{\cite{bpi1,bpi2}}
Let $\pi=(\pi_1,\pi_2,\cdots,\pi_n)\in\mathbb{R}^n$ be a nonnegative vector. For any multi-index $J=(i_2,\cdots,i_m)\in[n]^{m-1}$, define the product-form
$$\pi_J=\pi_{i_2\cdots i_m}:=\pi_{i_2}\cdots\pi_{i_m}.$$
Suppose $\pi$ satisfies the sum constraint
\begin{equation}\label{sumpi}
0<\sum\limits_{J\in[n]^{m-1}}\pi_{J}\leq1
\end{equation}
A tensor $\mathcal{A} = (a_{i_1i_2\cdots i_m})\in T_{m,n}$ is said to be a $B_{\pi}^R$-tensor with $R=R(\mathcal{A})$, if the following two conditions hold for all $i\in[n]$:
\begin{itemize}
  \item [(a)]  $R_i(\mathcal{A})>0$,
  \item [(b)]  $\pi_JR_{i}(\mathcal{A})>a_{iJ}\mbox{ for all }J\in[n]^{m-1} \setminus\{(i,\cdots,i)\}.$
\end{itemize}
\end{definition}
Throughout this paper, whenever no confusion can arise, the superscript $R$ in $B_{\pi}^R$ is understood to denote the row sum vector $R(\mathcal{A})$ of $\mathcal{A}$ under consideration.

\begin{remark}
In the case where $\pi_{i}=\frac{1}{n}$ for all $i\in[n]$, this class of a $B_{\pi}^R$-tensor coincides with the class of a $B$-tensor.
\end{remark}

\begin{definition}\label{def-pt}\textup{\cite{dlq2018}}
A tensor $\mathcal{A}\in T_{m,n}$ is called a $P$-tensor if for every nonzero vector $x\in\mathbb{R}^n$, there exists an index $i\in[n]$ such that
$$x_{i}^{m-1}(\mathcal{A}x^{m-1})_i>0.$$
\end{definition}

\begin{proposition}\label{p1}\textup{\cite{bpi2}}
Let $\mathcal{A} = (a_{i_1i_2\cdots i_m})\in T_{m,n}$ be a $B_{\pi}^R$-tensor. If $m$ is odd, then $\mathcal{A}$ is a $P$-tensor. If $m$ is even and $\mathcal{A}$ is symmetric, then $\mathcal{A}$ is also a $P$-tensor.
\end{proposition}
For any $J\in[n]^{m-1}$ and $\mathcal{A} = (a_{i_1i_2\cdots i_m})\in T_{m,n}$, define
$$M_{iJ}(\mathcal{A}):=\frac{a_{iJ}}{R_i(\mathcal{A})} \mbox{ for } i\in[n],$$
and
$$M_{J}(\mathcal{A}):=\max\limits_{\substack{i\in[n]\\ (i,\cdots,i)\neq J}} M_{iJ}(\mathcal{A}).$$

In Theorem 6.2 of \cite{bpi1}, the authors provided a characterization of $B_\pi^R$-tensors. Through a careful re-examination of their complete proof process, we find that the theorem implicitly uses the critical condition that, the value $\pi_J$ must be decomposable into the product of component values $\pi_{i_2}\pi_{i_3}\cdots\pi_{i_m}$ for any multi-index $J=(i_2,i_3,\cdots,i_m)\in[n]^{m-1}$. However, this core premise was completely omitted from the final theorem statement, resulting in the original condition being only necessary but not sufficient. To provide a rigorous and accurate characterization and avoid erroneous judgments in practical applications, we propose a corrected equivalent proposition and elaborate on the limitations of the original theorem through a constructive counterexample.

\begin{proposition}\label{p2}
Let $\mathcal{A} = (a_{i_1i_2\cdots i_m})\in T_{m,n}$ with $R_i(\mathcal{A})>0$ for all $i\in[n]$. Then $\mathcal{A}$ is a $B_{\pi}^R$-tensor if and only if the nonnegative vector $\pi$ satisfies condition \eqref{sumpi}, and for all $J\in[n]^{m-1}$,
$\pi_J>M_J(\mathcal{A}).$
\end{proposition}

\begin{example}Consider a tensor $\mathcal{A} = (a_{i_1i_2i_3})\in T_{3,2}$, with entries
 $$a_{112}=a_{121}=a_{212}=a_{221}=0.4,\ \ a_{122}=a_{211}=0.01,\ \ a_{111}=a_{222}=0.19.$$
 Then
$$R_1(\mathcal{A})=R_2(\mathcal{A})=1>0,$$
$$M_{11}(\mathcal{A})+M_{22}(\mathcal{A})+ M_{12}(\mathcal{A})+M_{21}(\mathcal{A})=0.82<1,$$
where $M_{11}(\mathcal{A})=M_{22}(\mathcal{A})=0.01$ and $M_{12}(\mathcal{A})=M_{21}(\mathcal{A})=0.4.$ Hence the original condition ($\sum\limits_{J\in[n]^{m-1}} M_J(\mathcal{A}) < 1$) would be satisfied.

However, if $\mathcal{A}$ is a $B_{\pi}^R$-tensor for some nonnegative vector $\pi=(\pi_1,\pi_2)$ satisfying \eqref{sumpi}, then $\pi$ would have to meet
$$\pi_1^2>0.01, \ \ \pi_1\pi_2>0.4,\ \ \pi_2^2>0.01.$$
Consequently,
$$\sum\limits_{J\in[2]^{2}}\pi_J=\pi_1^2+2\pi_1\pi_2+\pi_2^2=(\pi_1+\pi_2)^2\geq4\pi_1\pi_2>4\times0.4=1.6>1,$$
which contradicts condition \eqref{sumpi}. Hence no such $\pi$ exists, and therefore $\mathcal{A}$ is not a $B_\pi^R$-tensor.
\end{example}
This counterexample shows that the condition $\sum\limits_{J\in[n]^{m-1}} M_J(\mathcal{A}) < 1$ is not sufficient, the corrected condition, which is existence of a product-form $\pi$ with $\pi_J > M_J(\mathcal{A})$ and sum constraint, is necessary and sufficient.

\begin{definition}\label{def-dd}\textup{\cite{qs2014}}
A tensor $\mathcal{A} = (a_{i_1i_2\cdots i_m})\in T_{m,n}$ is called diagonally dominant if for all $i\in[n]$,
$$a_{ii\cdots i}\geq\sum\limits_{J\in[n]^{m-1}\setminus\{(i,\cdots,i)\}}|a_{iJ}|.$$
It is diagonally dominant if the inequality is strict for all $i\in[n]$.
\end{definition}

\begin{definition}\label{def-zt}\textup{\cite{zq2014}}
A tensor $\mathcal{A} = (a_{i_1i_2\cdots i_m})\in T_{m,n}$ is called a $Z$-tensor if all of its off-diagonal entries are non-positive, that is, $a_{iJ}\leq0$ whenever $J\in[n]^{m-1}\setminus\{(i,\cdots,i)\}$.
\end{definition}

\begin{proposition}\label{p4}\textup{\cite{bpi2}}
If $\mathcal{A} = (a_{i_1i_2\cdots i_m})\in T_{m,n}$ is a $Z$-tensor. Then $\mathcal{A}$ is a $B_{\pi}^R$-tensor with any positive vector $\pi\in\mathbb{R}^n$ satisfying condition $\eqref{sumpi}$ if and only if $\mathcal{A}$ is strictly diagonally dominant.
\end{proposition}

\begin{proposition}\label{p23}\textup{\cite{bpi1}}
Let nonnegative vector $\pi\in\mathbb{R}^n$ satisfy condition $\eqref{sumpi}$, and let $\mathcal{A}\in T_{m,n}$ be a $B_{\pi}^R$-tensor. Then for all $i\in[n]$,
\begin{itemize}
  \item [(a)] $a_{ii\cdots i}>\pi_{i\cdots i}R_i(\mathcal{A})$;
  \item [(b)] $a_{ii\cdots i}>a_{iJ}$ if $\pi_{i\cdots i}\geq \pi_{J}$, $J\in[n]^{m-1}\setminus\{(i,\cdots,i)\}$.
\end{itemize}
\end{proposition}

\begin{proposition}\label{p24}\textup{\cite{bpi1}}
Let nonnegative vector $\pi\in\mathbb{R}^n$ satisfy condition $\eqref{sumpi}$, and let $\mathcal{A}\in T_{m,n}$ be a $B_{\pi}^R$-tensor. If $\alpha\in\mathbb{R}^n$ also satisfies condition $\eqref{sumpi}$ and $\alpha\geq\pi$, then $\mathcal{A}$ is also a $B_{\alpha}^R$-tensor.
\end{proposition}

Let $\underline{\mathcal{A}}=(\underline{a}_{i_1i_2\cdots i_m})$ and $\overline{\mathcal{A}}=(\overline{a}_{i_1i_2\cdots i_m})\in T_{m,n}$ with $\underline{\mathcal{A}}\leq\overline{\mathcal{A}}$, where $\underline{\mathcal{A}}\leq\overline{\mathcal{A}}$ are to be understood componentwise. The set of tensors
$$\mathcal{A}^I=[\underline{\mathcal{A}},\overline{\mathcal{A}}]= \{\mathcal{A}:\underline{\mathcal{A}}\leq \mathcal{A}\leq \overline{\mathcal{A}}\}$$
is called an interval tensor.

Denote the center and the radius of interval by
$$\mathcal{A}^c=(a^c_{i_1i_2\cdots i_m})=\frac{\underline{\mathcal{A}}+\overline{\mathcal{A}}}{2}, \mbox{ and }\Delta=(\delta_{i_1i_2\cdots i_m})=\frac{\overline{\mathcal{A}}-\underline{\mathcal{A}}}{2},$$
 so that $\mathcal{A}^I=[\mathcal{A}^c- \Delta, \mathcal{A}^c+ \Delta]$. Clearly, $\Delta$ is a nonnegative tensor, i.e., all its entries $\delta_{i_1i_2\cdots i_m}\geq0$. An interval tensor $\mathcal{A}^I$ is called symmetric if both $\mathcal{A}^c$ and $\Delta$ are symmetric. It is worth mentioning that an interval tensor $\mathcal{A}^I$ is symmetric does not means each tensor $\mathcal{A}\in\mathcal{A}^I$ is symmetric.

\begin{definition}\textup{\cite{bf2022,sr2020,boz2020}}
An interval tensor $\mathcal{A}^I$ is called an interval $Z$-tensor (resp. interval $P$-tensor, positive definite interval tensor) if every tensor $\mathcal{A}\in\mathcal{A}^I$ is a $Z$-tensor (resp. $P$-tensor, positive definite tensor).
\end{definition}

Write
$$\mathcal{A}^z:=\mathcal{A}^c-\Delta\times_1 T_z\times_2 T_z\times_3\cdots\times_m T_z,$$
where $T_z$ is the $n\times n$ diagonal matrix with diagonal entries taken from a vector
$$z\in Y= \{z\in \mathbb{R}^n: |z_j|=1 \mbox{ for } j\in[n]\}.$$

\begin{lemma}\label{pip}\textup{\cite{bf2022}}
An interval tensor $\mathcal{A}^I$ in $T_{m,n}$ is an interval $P$-tensor if and only if $\mathcal{A}^z$ is a $P$-tensor for any $z\in Y$.
\end{lemma}

Inspired by the existing definition of various classes of interval structured tensors, and the corresponding notion of interval $B_{\pi}^R$-matrix (see \cite{lm2023}), we analogously introduced the definitions of an interval $B_{\pi}^R$-tensor.

\begin{definition}
Let $\mathcal{A}^I$ be an interval tensor in $T_{m,n}$, and let nonnegative vector $\pi\in\mathbb{R}^n$ satisfy condition \eqref{sumpi}. Denote $R^I=[R(\underline{\mathcal{A}}),R(\overline{\mathcal{A}})]=\{R(\mathcal{A}):\mathcal{A}\in\mathcal{A}^I\}$.
The interval tensor $\mathcal{A}^I$ is called an interval $B_{\pi}^{R^I}$-tensor if every tensor $\mathcal{A}\in\mathcal{A}^I$ is a $B_{\pi}^R$-tensor with row sum vector $R=R(\mathcal{A})\in R^I$.
\end{definition}

\section{$B_{\pi}^R$-Tensors}
\setcounter{section}{3}
We first derive several fundamental properties that reveal the structural characteristics of this tensor class.
\begin{proposition}\label{th2}
Let nonnegative vector $\pi\in\mathbb{R}^n$ satisfy condition $\eqref{sumpi}$ and let $\mathcal{A}\in T_{m,n}$ be a $B_{\pi}^R$-tensor. Then
\begin{itemize}
  \item [(a)] $a_{kk\cdots k}>a_{kJ}$ for all $J\in[n]^{m-1}\setminus\{(k,\cdots,k)\}$, where $k=\arg\max\limits_{i\in[n]}\{\pi_i\}$;
  \item [(b)] $a_{iJ}<0$ whenever $\pi_J=0$, for all $J\in[n]^{m-1}\setminus\{(i,\cdots,i)\}$.
\end{itemize}
\end{proposition}
\begin{proof}
$(a)$ follows directly from $(b)$ of Proposition \ref{p23}. Next, we prove $(b)$ of this Proposition.

Because $\mathcal{A}$ is a $B_{\pi}^R$-tensor, for every $i\in[n]$, we have
$$R_i(\mathcal{A})>0,$$
and for every $J\in[n]^{m-1}\setminus\{(i,\cdots,i)\}$, there is
$$a_{iJ}<\pi_{J}R_i(\mathcal{A}).$$
If $\pi_J=0$, the right-hand side is non-positive, hence $a_{iJ}<0$.
\end{proof}

While Proposition \ref{p2} provides an equivalent characterization of $B_{\pi}^{R}$-tensors, it requires a pre-specified vector $\pi$ satisfying the product-form condition, which is often not readily available in practical applications. To address this limitation, we establish a constructive sufficient condition that allows us to explicitly construct such a vector $\pi$ when it exists.

\begin{theorem}\label{p3}
Let $\mathcal{A} = (a_{i_1i_2\cdots i_m})\in T_{m,n}$ with $R_i(\mathcal{A})>0$ for all $i\in[n]$. If
$$S:=\sum\limits_{i=1}^n[\max\limits_{\{J\in[n]^{m-1}:i\in J\}}M_J]_+^{\frac{1}{m-1}}<1,$$
 where $[a]_+=\max\{0,a\}$, then there exists a constructible positive vector $\pi=(\pi_1,\pi_2,\cdots,\pi_n)\in\mathbb{R}^n$ satisfying condition \eqref{sumpi} such that $\mathcal{A}$ is a $B_{\pi}^R$-tensor.
\end{theorem}
\begin{proof}
For each $i\in[n]$, choose any $0<\varepsilon\leq\frac{1-S}{n}$, and set
$$\pi_i=[\max\limits_{\{J\in[n]^{m-1}:i\in J\}}M_J]_+^{\frac{1}{m-1}}+\varepsilon>0.$$
(For instance, one may simply take $\varepsilon=\frac{1-S}{n}$). Then
$$\sum_{i=1}^n \pi_i=S+n\cdot\frac{1-S}{n}=1,$$
and consequently
$$\sum\limits_{J\in[n]^{m-1}} \pi_{J}=\sum_{i_2,\cdots,i_m=1}^n \pi_{i_2}\cdots \pi_{i_m}=(\sum_{i=1}^n \pi_i)^{m-1}=1,$$
which means that $\pi$ satisfies condition \eqref{sumpi}.

Now take any $i\in[n]$, and any $J=(i_2,\cdots,i_m)\in[n]^{m-1}\setminus\{(i,\cdots,i)\}$. We have
\begin{align}\label{pj}
\pi_{J}=\pi_{i_2}\cdots\pi_{i_m}>\prod\limits_{k=2}^m[\max\limits_{\{K\in[n]^{m-1}:i_k\in K\}}M_K]_+^{\frac{1}{m-1}}\geq M_J,
\end{align}
and therefore
$$\pi_{J}R_i(\mathcal{A})>a_{iJ}.$$
Hence $\mathcal{A}$ is a $B_{\pi}^R$-tensor.
\end{proof}
\begin{remark}
In the construction of $\pi$ in Theorem \ref{p3}, the parameter $\varepsilon$ cannot be taken as $0$. If $\varepsilon=0$, then there may exists some multi-index $J$ such that the strict inequality in \eqref{pj} may not hold, which violates the condition $(b)$ of Definition \ref{def-bt}.
\end{remark}
In particular, for nonnegative tensors, the above condition can be simplified since all $M_{J}$ values are automatically nonnegative, eliminating the need for the positive part operator.
\begin{corollary}\label{c3.1}
Let 
$\mathcal{A} = (a_{i_1i_2\cdots i_m})\in T_{m,n}$ be a nonnegative tensor with $R_i(\mathcal{A})>0$ for all $i\in[n]$. If $$S_0:=\sum\limits_{i=1}^n(\max\limits_{\{J\in[n]^{m-1}:i\in J\}}M_J)^{\frac{1}{m-1}}<1,$$
then there exists a constructible positive vector $\pi=(\pi_1,\pi_2,\cdots,\pi_n)\in\mathbb{R}^n$ satisfying condition \eqref{sumpi} such that $\mathcal{A}$ is a $B_{\pi}^R$-tensor.
\end{corollary}
\begin{proof}
Since $\mathcal{A}$ is nonnegative, $M_J\geq0$ for all $J\in[n]^{m-1}$. Hence
$$[\max\limits_{\{J\in[n]^{m-1}:i\in J\}}M_J]_+^{\frac{1}{m-1}}=(\max\limits_{\{J\in[n]^{m-1}:i\in J\}}M_J)^{\frac{1}{m-1}}.$$
It follows that $S=S_0<1$. The conclusion follows directly from Theorem \ref{p3}.
\end{proof}

\begin{example}
Let $\mathcal{A} \in T_{3,2}$ be defined by
$$a_{111} = 90, \quad a_{112} = a_{121} = 5, \quad a_{122} = 0,$$
$$a_{222} = 30, \quad a_{211} = 60, \quad a_{212} = a_{221} = 5.$$
The row sums are
$$R_1(\mathcal{A}) =R_2(\mathcal{A}) = 100 > 0.$$
Compute the $M_J$ values as following.
$$M_{11}=\frac{a_{211}}{R_2}=\frac{60}{100}=\frac{3}{5},$$ $$M_{12}=M_{21}=\max\left\{\frac{a_{112}}{R_1},\frac{a_{212}}{R_2}\right\}=\frac{5}{100}=\frac{1}{20},$$
$$M_{22}=\frac{a_{122}}{R_1}=0.$$
Then
$$S_0=\sqrt{\max_{J\ni1}M_J}+\sqrt{\max_{J\ni2}M_J}=\sqrt{\frac{3}{5}}+\sqrt{\frac{1}{20}}=\frac{2\sqrt{15}+\sqrt{5}}{10}<1.$$
By Corollary \ref{c3.1}, we construct the vector
\begin{align*}
\pi_1 &= \sqrt{\frac{3}{5}} + \frac{1-S_0}{2} = \frac{10 + 2\sqrt{15} - \sqrt{5}}{20}, \\
\pi_2 &= \sqrt{\frac{1}{20}} + \frac{1-S_0}{2} = \frac{10 - 2\sqrt{15} + \sqrt{5}}{20}.
\end{align*}
which satisfies $\sum_{J\in[2]^2}\pi_J=1$. It is straightforward to verify that $\pi_J>M_J$ for all $J\in[2]^2$, confirming that $\mathcal{A}$ is a $B_{\pi}^{R}$-tensor.
\end{example}
This example also shows that there exists a $B_{\pi}^{R}$-tensor that is not a B-tensor.
Since $a_{211} = 60 >\frac{100}{2}=50$ in this example, which violates the B-tensor condition.
This example highlights the greater flexibility of $B_{\pi}^{R}$-tensors, that is the adjustable vector $\pi$ allows larger off-diagonal entries in specific positions, expanding their applicability.

For symmetric even-order tensors, a tensor is a $P$-tensor implies it is positive definite \cite{sq2015}. Then by Proposition \ref{p1}, we obtain a sufficient condition of positive definiteness for a symmetric tensor, which is equivalent to the positivity of its H-eigenvalues \cite{q2005}.

\begin{corollary}\label{c3.2}
Let $m$ be even, let $\mathcal{A} = (a_{i_1i_2\cdots i_m})\in T_{m,n}$ be a symmetric tensor, and let $R_i(\mathcal{A})>0$ for all $i\in[n]$. If $S<1$, then $\mathcal{A}$ is positive definite, and consequently all H-eigenvalues of $\mathcal{A}$ are positive.
\end{corollary}
\begin{proof}
By Theorem \ref{p3}, the condition $S<1$ guarantees that $\mathcal{A}$ is a $B_{\pi}^R$-tensor for some positive $\pi$. Since $m$ is even and $\mathcal{A}$ is symmetric, Proposition \ref{p1} implies that $\mathcal{A}$ is a $P$-tensor. Moreover, $\mathcal{A}$ is positive definite, and all H-eigenvalues of $\mathcal{A}$ are positive.
\end{proof}

Next, we investigate the convex combination property of $B_{\pi}^{R}$-tensors, which plays a crucial role in analyzing the closure properties of structured tensor classes under linear operations.

\begin{theorem}\label{th3}
Suppose nonnegative vectors $\pi,\alpha\in\mathbb{R}^n$ both satisfy condition $\eqref{sumpi}$, and let $\mathcal{A},\mathcal{B}\in T_{m,n}$ be a $B_{\pi}^R$-tensor and a $B_{\alpha}^R$-tensor, respectively, with $R=R(\mathcal{A})$ and $R_i(\mathcal{A})=R_i(\mathcal{B})$ for all $i\in[n]$. Let $s,t\geq0$ with $s+t>0$, and $\pi'=(\pi'_1,\pi'_2,\cdots,\pi'_n)\in\mathbb{R}^n$ by $\pi'_i\geq\max\{\pi_i,\alpha_i\}$ for all $i\in[n]$. If $\pi'$ satisfies condition \eqref{sumpi}, then $s\mathcal{A}+t\mathcal{B}$ is a $B_{\pi'}^{R}$-tensor with $R(s\mathcal{A}+t\mathcal{B})=(s+t)R$.
\end{theorem}
\begin{proof}
For each $i\in[n]$,
\begin{align*}
R_i(s\mathcal{A}+t\mathcal{B})=&\sum\limits_{J\in[n]^{m-1}}(sa_{iJ}+tb_{iJ})\\
=&(s\sum\limits_{J\in[n]^{m-1}}a_{iJ}+t\sum\limits_{J\in[n]^{m-1}}b_{iJ})\\
=&sR_i(\mathcal{A})+tR_i(\mathcal{B})\\
=&(s+t)R_i(\mathcal{A})>0.
\end{align*}
Since $\pi'_i\geq\max\{\pi_i,\alpha_i\}\geq0$ for all $i\in[n]$, for any $J=(i_2,\cdots,i_m)\in[n]^{m-1}\setminus\{(i,\cdots,i)\}$ we obtain
\begin{align*}
\pi'_{J}=&\prod\limits_{k=2}^m\pi'_{i_k}\\
\geq&\prod\limits_{k=2}^m\max\{\pi_{i_k},\alpha_{i_k}\}\\
\geq&\max\{\prod\limits_{k=2}^m\pi_{i_k},\prod\limits_{k=2}^m\alpha_{i_k}\}\\
=&\max\{\pi_{J},\alpha_{J}\}.
\end{align*}
Therefore,
\begin{align*}
\pi'_{J}(s+t)R_i(\mathcal{A})\geq&\max\{\pi_{J},\alpha_{J}\}(s+t)R_i(\mathcal{A})\\
\geq&s\pi_{J}R_i(\mathcal{A})+t\alpha_{J}R_i(\mathcal{B})\\
>&sa_{iJ}+tb_{iJ}.
\end{align*}
In conclusion, $s\mathcal{A}+t\mathcal{B}$ a $B_{\pi'}^{R}$-tensor with $R(s\mathcal{A}+t\mathcal{B})=(s+t)R$ when $\pi'$ satisfies condition \eqref{sumpi}.
\end{proof}

\begin{corollary}\label{c3.3}
If $\alpha=\pi$ in Theorem \ref{th3}, then $s\mathcal{A}+t\mathcal{B}$ is a $B_{\pi}^{R}$-tensor with $R(s\mathcal{A}+t\mathcal{B})=(s+t)R$.
\end{corollary}

\begin{theorem}\label{th4}
Let $\pi\in\mathbb{R}^n$ be a positive vector satisfying condition $\eqref{sumpi}$, and let $\mathcal{A}\in T_{m,n}$ be a $B_{\pi}^R$-tensor with $R=R(\mathcal{A})$. Increasing the diagonal entries of $\mathcal{A}$ yields a tensor $\mathcal{B}$ that is a $B_{\pi}^{Q}$, where $Q=R(\mathcal{B})$.
\end{theorem}
\begin{proof}
By the definition of $\mathcal{B}$, there exists a diagonal tensor $\mathcal{D}=\mbox{diag}(d_{1},d_{2},\cdots,d_{n})$ with $d_{i}\geq0$ for all $i\in[n]$ such that $\mathcal{B}=\mathcal{A}+\mathcal{D}$. Then for every $i\in[n]$,
$$R_i(\mathcal{B})=R_i(\mathcal{A})+d_i>0,$$
and for all $J\in[n]^{m-1}\setminus\{(i,\cdots,i)\}$,
$$b_{iJ}=a_{iJ}.$$
Consequently, 
$$b_{iJ}<\pi_{J}R_i(\mathcal{A})\leq\pi_{J}R_i(\mathcal{B}).$$
Hence $\mathcal{B}$ is a $B_{\pi}^{Q}$-tensor.
\end{proof}

Lastly, we examine the invariance property of $B_{\pi}^{R}$-tensors with respect to index permutations

\begin{proposition}\label{thp}
Let nonnegative vector $\pi\in\mathbb{R}^n$ satisfy condition \eqref{sumpi}, and $\sigma$ be a permutation of $[m]$ satisfying $\sigma(1) = 1$.
If $\mathcal{A}= (a_{i_1i_2\cdots i_m})\in T_{m,n}$ is a $B_{\pi}^R$-tensor, then $\mathcal{A}^\sigma$ is also a $B_{\pi}^R$-tensor, where $\mathcal{A}^\sigma$ is defined by
$$a^\sigma_{i_1i_2\cdots i_m} = a_{i_{\sigma(1)}i_{\sigma(2)}\cdots i_{\sigma(m)}}.$$
\end{proposition}
\begin{proof}
Since $\mathcal{A}$ is a $B_{\pi}^R$-tensor, then $R_i(\mathcal{A}) >0$ for all $i\in[n]$.

By the condition $\sigma(1) = 1$, the permutation only rearranges the order of the last $m-1$ indices in each entry of the $i$-th row. The set of entries in the $i$-th row of $\mathcal{A}^\sigma$ is identical to that of the $i$-th row of $\mathcal{A}$, with only the summation order changed. Since finite summation is commutative, we have
$$R_i(\mathcal{A}^\sigma) = \sum_{J \in [n]^{m-1}} a^\sigma_{iJ} = \sum_{J \in [n]^{m-1}} a_{iJ} = R_i(\mathcal{A}) > 0.$$
This confirms that condition $(a)$ of Definition \ref{def-bt} holds for $\mathcal{A}^\sigma$.

For any $i \in [n]$ and any $J\in[n]^{m-1}\setminus\{(i,\cdots,i)\}$, let $J' = \sigma(J)$ be the multi-index obtained by permuting $J$ according to $\sigma$. Then
$$a^\sigma_{iJ} = a_{iJ'}.$$
 Since $\sigma$ only permutes the last $m-1$ indices, we have
 $$\pi_J = \pi_{i_2} \pi_{i_3} \cdots \pi_{i_m} = \pi_{i_{\sigma(2)}} \pi_{i_{\sigma(3)}} \cdots \pi_{i_{\sigma(m)}} = \pi_{J'}.$$
 Because $\mathcal{A}$ is a $B_{\pi}^R$-tensor, it holds that
 $$\pi_J R_i(\mathcal{A}^\sigma)
 = \pi_{J'} R_i(\mathcal{A})
 > a_{iJ'}
 = a^\sigma_{iJ}.$$
Therefore $\mathcal{A}^\sigma$ satisfies condition $(b)$ of Definition \ref{def-bt} and is a $B_{\pi}^R$-tensor.
\end{proof}


\section{Interval $B_{\pi}^{R^I}$-Tensors}
\setcounter{section}{4}

Having established the fundamental properties of deterministic $B_{\pi}^{R}$-tensors, we now extend our analysis to the interval setting, which is motivated by the prevalence of uncertain data in real-world applications. First, we establish the relationship between interval $B_{\pi}^{R^{I}}$-tensors and interval $P$-tensors, which follows from the corresponding result for deterministic tensors.

\begin{corollary}\label{c4.1}
Let $\mathcal{A}^I$ be an interval $B_{\pi}^{R^I}$-tensor in $T_{m,n}$ with some $\pi\geq0$. Then $\mathcal{A}^I$ is an interval $P$-tensor, provided that either $m$ is odd, or $m$ is even and $\mathcal{A}^I$ is symmetric.
\end{corollary}
\begin{proof}
For odd $m$, the conclusion follows immediately from Proposition \ref{p1} and the definition of an interval $P$-tensor. It therefore suffices to consider the case where $m$ is even.

Assume that $\mathcal{A}^I$ is an even-order symmetric interval $B_{\pi}^{R^I}$-tensor, then for every $z\in Y$, the tensor $\mathcal{A}^z$ belongs to $\mathcal{A}^I$, and it is an even-order symmetric $B_{\pi}^{R}$-tensor by the definition of $\mathcal{A}^z$. Hence each $\mathcal{A}^z\in\mathcal{A}^I$ is a $P$-tensor by Proposition \ref{p1}.
According to Lemma \ref{pip}, $\mathcal{A}^I$ is an interval $P$-tensor.
\end{proof}

Next, we derive several fundamental properties of interval $B_{\pi}^{R^{I}}$-tensors that reveal the inherent relationships between their lower and upper endpoint tensors.
\begin{proposition}\label{p4.1}
Let $\mathcal{A}^I$ be an interval $B_{\pi}^{R^I}$-tensor in $T_{m,n}$ with some $\pi\geq0$. Then for $i\in[n]$,
\begin{itemize}
  \item [(a)] $\underline{a}_{ii\cdots i}>\max\{\pi_i^{m-1}R_i(\underline{\mathcal{A}}),\pi_i^{m-1}(\underline{a}_{ii\cdots i}+\sum\limits_{\substack{J\in[n]^{m-1}\\
      J\neq(i,\cdots,i)}}\overline{a}_{iJ})\}$;
  \item [(b)] $\underline{a}_{ii\cdots i}>\overline{a}_{iJ}$ when $\pi_i^{m-1}>\pi_J$, where $J\in[n]^{m-1}\setminus\{(i,\cdots,i)\}$;
  \item [(c)] $\underline{a}_{kk\cdots k}>\overline{a}_{kJ}$ when $k=\mbox{arg}\max\limits_{j\in[n]}\{\pi_j\}$, where $J\in[n]^{m-1}\setminus\{(k,\cdots,k)\}$;
  \item [(d)] $\overline{a}_{iJ}<0$ when $\pi_{J}=0$, where $J\in[n]^{m-1}\setminus\{(i,\cdots,i)\}$.
\end{itemize}
\end{proposition}
\begin{proof}
$(a)$ Since $\underline{A}\in\mathcal{A}^I$ is a $B_{\pi}^{R}$-tensor. By $(a)$ of Proposition \ref{p23}, we can obtain that
$$\underline{a}_{ii\cdots i}>\pi_i^{m-1}R_i(\underline{A})$$
for all $i\in [n]$.

Define $\mathcal{A}^{(1)}=(a_{i_1i_2\cdots i_m}^{(1)})\in T_{m,n}$ by
$$a^{(1)}_{i_1i_2\cdots i_m}=\begin{cases}
\underline{a}_{i_1i_1\cdots i_1},\mbox{ if }(i_2,i_3,\cdots,i_m)= (i_1,\cdots,i_1),\\
\overline{a}_{i_1i_2\cdots i_m},\mbox{ otherwise}.
\end{cases}$$
Obviously $\mathcal{A}^{(1)}\in\mathcal{A}^I$, then $\mathcal{A}^{(1)}$ is also a $B_{\pi}^R$-tensor. By $(a)$ of Proposition \ref{p23}, we can obtain that
$$\underline{a}_{ii\cdots i}>\pi_i^{m-1}R_i(\mathcal{A}^{(1)})=\pi_i^{m-1}(\underline{a}_{ii\cdots i}+\sum\limits_{\substack{J\in[n]^{m-1}\\
      J\neq(i,\cdots,i)}}\overline{a}_{iJ})$$
for all $i\in [n]$.

$(b)$ Suppose $\pi_i^{m-1}>\pi_J$ for $i\in[n]$ and $J\in[n]^{m-1}\setminus\{(i,\cdots,i)\}$. Consider the tensor $\mathcal{A}^{(1)}$ again, and according to $(a)$ of Proposition \ref{p23}, we can obtain that
$$\underline{a}_{ii\cdots i}>\overline{a}_{iJ}.$$

$(c)$ This is a direct consequence of $(b)$.

$(d)$ As $\overline{A}\in\overline{A}^I$ is a $B_{\pi}^R$-tensor. If $\pi_{J}=0$, where $J\in[n]^{m-1}\setminus\{(i,\cdots,i)\}$ and $i\in[n]$, by $(b)$ of Proposition \ref{th2}, we can obtain that
$$\overline{a}_{iJ}<0.$$
\end{proof}

Analogous to the deterministic case, interval $B_{\pi}^{R^{I}}$-tensors also exhibit monotonicity with respect to the nonnegative vector $\pi$.

\begin{proposition}\label{p27}
Let $\mathcal{A}^I$ be an interval $B_{\pi}^{R^I}$-tensor in $T_{m,n}$, where nonnegative vector $\pi\in\mathbb{R}^n$ satisfies condition \eqref{sumpi}. If $\alpha\in\mathbb{R}^n$ satisfies condition \eqref{sumpi} and $\alpha\geq\pi$, then $\mathcal{A}^I$ is also an interval $B_{\alpha}^{R^I}$-tensor.
\end{proposition}
\begin{proof}
As $\mathcal{A}^I$ is an interval $B_{\pi}^{R^I}$-tensor, all $\mathcal{A}\in\mathcal{A}^I$ are $B_{\pi}^{R}$-tensors.
If $\alpha\geq\pi$, from Proposition \ref{p24}, all $\mathcal{A}\in\mathcal{A}^I$ are $B_{\alpha}^{R}$-tensors.
Hence, $\mathcal{A}^I$ is an interval $B_{\alpha}^{R^I}$-tensor.
\end{proof}

To provide practical verifiable criteria for interval $B_{\pi}^{R^{I}}$-tensors, we establish an equivalent characterization based solely on inequalities involving the entries of the lower and upper endpoint tensors. It allows us to check the interval $B_{\pi}^{R^{I}}$ property without enumerating all tensors in the interval.

\begin{theorem}\label{th1}
Let $\mathcal{A}^I$ be an interval tensor in $T_{m,n}$, and let nonnegative vector $\pi\in\mathbb{R}^n$ satisfy condition \eqref{sumpi}. Then $\mathcal{A}^I$ is an interval $B_{\pi}^{R^I}$-tensor if and only if
\begin{itemize}
  \item [(a)] $R_i(\underline{\mathcal{A}})>0$ for all $i\in[n]$;
  \item [(b)] for every $i\in[n]$ and every $J\in[n]^{m-1}\setminus\{(i,\cdots,i)\}$,
  \begin{itemize}
    \item [(b$_1$)] $\underline{a}_{ii\cdots i}+\sum\limits_{\substack{K\in[n]^{m-1}, K\neq J\\ K\neq(i,\cdots,i)}}\underline{a}_{iK}>(\frac{1}{\pi_J}-1)\overline{a}_{iJ}$, when $0<\pi_J\leq1$;
    \item [(b$_2$)] $0>\overline{a}_{iJ}$, when $\pi_J=0$.
  \end{itemize}
\end{itemize}
\end{theorem}
\begin{proof} ``{\bf Necessity.}'' $(a)$ Because $\underline{\mathcal{A}}=(\underline{a}_{i_1i_2\cdots i_m})\in\mathcal{A}^I$, it is a $B_{\pi}^R$-tensor with $R=R(\underline{\mathcal{A}})$. Condition $(a)$ of Definition \ref{def-bt} immediately gives $R_i(\underline{\mathcal{A}})>0$ for all $i\in [n]$.

$(b)$ Let $i\in[n]$ and $J\in[n]^{m-1}\setminus\{(i,\cdots,i)\}$.

$(b_1)$  Case $0<\pi_J\leq1$: Define $\mathcal{A}^{(0)}=(a_{i_1i_2\cdots i_m}^{(0)})\in T_{m,n}$ by
$$a^{(0)}_{i_1i_2\cdots i_m}=\begin{cases}
\overline{a}_{i_1J},\mbox{ if }J\neq (i_1,\cdots,i_1) \mbox{ and  }(i_2,i_3,\cdots,i_m)=J ,\\
\underline{a}_{i_1i_2\cdots i_m},\mbox{ otherwise}.
\end{cases}$$
Clearly $\mathcal{A}^{(0)}\in\mathcal{A}^I$, so it is a $B_{\pi}^R$-tensor with $R=R({\mathcal{A}^{(0)}})$. Hence,
$$\pi_J(\underline{a}_{ii\cdots i}+\overline{a}_{iJ}+\sum\limits_{\substack{K\in[n]^{m-1}, K\neq J\\ K\neq(i,\cdots,i)}}\underline{a}_{iK})=\pi_JR_i(\underline{A}^{(0)})>a^{(0)}_iJ=\overline{a}_{iJ},$$
Which gives
$$\underline{a}_{ii\cdots i}+\sum\limits_{\substack{K\in[n]^{m-1}, K\neq J\\ K\neq(i,\cdots,i)}}\underline{a}_{iK}>(\frac{1}{\pi_J}-1)\overline{a}_{iJ}.$$

$(b_2)$ Case $\pi_J=0$: Now consider $\overline{\mathcal{A}}=(\overline{a}_{i_1i_2\cdots i_m})\in \mathcal{A}^I$. As it is a $B_{\pi}^R$-tensor with $R=R(\overline{\mathcal{A}})$,
$$0=\pi_JR_i(\overline{\mathcal{A}})>\overline{a}_{iJ},$$
so $(b_2)$ holds.

``{\bf Sufficiency.}'' Condition $(a)$ guarantees that for every $\mathcal{A}\in\mathcal{A}^I$ and every $i\in[n]$, $$R_i(\mathcal{A})\geq R_i(\underline{\mathcal{A}})>0,$$
i.e., part $(a)$ of Definition \ref{def-bt} holds.

Fix $i\in[n]$ and $J\in[n]^{m-1}\setminus\{(i,\cdots,i)\}$. For any $\mathcal{A}=(a_{i_1i_2\cdots i_m})\in\mathcal{A}^I$, we have
.$$\underline{a}_{iK}\leq a_{iK}\leq \overline{a}_{iK}$$
for all $K\in[n]^{m-1}$

If $0<\pi_J\leq1$, then $\frac{1}{\pi_J}-1\geq0$. Using $(b_1)$,
$$a_{ii\cdots i}+\sum\limits_{\substack{K\in[n]^{m-1}, K\neq J\\ K\neq(i,\cdots,i)}}a_{iK}\geq \underline{a}_{ii\cdots i}+\sum\limits_{\substack{K\in[n]^{m-1}, K\neq J\\ K\neq(i,\cdots,i)}}\underline{a}_{iK}>(\frac{1}{\pi_J}-1)\overline{a}_{iJ}\geq (\frac{1}{\pi_J}-1)a_{iJ},$$
and again we obtain $\pi_JR_i(\mathcal{A})>a_{iJ}$.

If $\pi_J=0$, $(b_2)$ gives $0>\overline{a}_{iJ}\geq a_{iJ}$, so
$$\pi_JR_i(\mathcal{A})=0> a_{iJ}.$$

Thus $(b)$ of Definition \ref{def-bt} holds for all $\mathcal{A}\in\mathcal{A}^I$, completing the proof.
\end{proof}

\begin{example}\label{e4.1}
Let $\pi=(0.4,0.6)$, $\mathcal{A}^I=[\underline{\mathcal{A}},\overline{\mathcal{A}}]$, where $\underline{a}_{111}=\underline{a}_{222}=\overline{a}_{111}=\overline{a}_{222}=2$, $\underline{a}_{122}=\underline{a}_{211}=\overline{a}_{122}=\overline{a}_{211}=0.1$, $\underline{a}_{112}=\underline{a}_{121}=\underline{a}_{221}=\underline{a}_{212}=0$ and $\overline{a}_{112}=\overline{a}_{121}=\overline{a}_{221}=\overline{a}_{212}=0.5$.
Then,
$$\sum\limits_{J\in[2]^2}\pi_J=\pi_1^2+2\pi_1\pi_2+\pi_2^2=1,$$
which implies that $\pi$ satisfies condition $\eqref{sumpi}$.
$$R_1(\underline{A})=2+0.1=2.1>0,\ \ R_2(\underline{A})=2+0.1=2.1>0,$$
then condition $(a)$ of Theorem \ref{th1} is satisfied.
For $J=(1,2)$, $\pi_J=0.4\times0.6=0.24\in(0,1]$, take $i=1$, then
$$\underline{a}_{111}+\sum\limits_{\substack{K\neq J\\ K\neq(1,1)}}\underline{a}_{1K}=2+0.1=2.1,$$
and
$$(\frac{1}{\pi_J}-1)\overline{a}_{112}=(\frac{1}{0.24}-1)\times 0.5<2.1,$$
satisfies condition ($b_1$) of Theorem \ref{th1}. Similarly, it can be verified that condition $(b)$ of Theorem \ref{th1} hold for all $i\in[2]$ and $J\in[2]^2$, therefore, by Theorem \ref{th1} $\mathcal{A}^I$ is an interval $B_{\pi}^{R^I}$-tensor.
\end{example}

The above theorem can be equivalently reformulated in terms of checking the $B_{\pi}^{R}$ property for a finite set of extremal tensors, which provides an alternative verification approach.

\begin{corollary}\label{c41}
Let $\mathcal{A}^I$ be an interval tensor in $T_{m,n}$ with $R_i(\underline{\mathcal{A}})>0$ for all $i\in[n]$, and let nonnegative vector $\pi\in\mathbb{R}^n$ satisfy condition \eqref{sumpi}. For each $J\in[n]^{m-1}$, define the tensor $\mathcal{A}^{(J)}=(a_{i_1i_2\cdots i_m}^{(J)})\in T_{m,n}$ as follows:
\begin{itemize}
  \item [(a)]  When $0<\pi_J\leq1$, let
   $$a^{(J)}_{i_1i_2\cdots i_m}=\begin{cases}
\overline{a}_{i_1J},\mbox{ if }J\neq (i_1,\cdots,i_1) \mbox{ and  }(i_2,i_3,\cdots,i_m)=J ,\\
\underline{a}_{i_1i_2\cdots i_m},\mbox{ otherwise}.
\end{cases}$$
  \item [(b)] When $\pi=0$, set $A^{(J)}=\overline{A}$.
\end{itemize}
Then $\mathcal{A}^I$ is an interval $B_{\pi}^{R^I}$-tensor if and only if every $A^{(J)}$ is a $B_{\pi}^{R}$-tensor.
\end{corollary}

If the vector $\pi>0$ satisfying condition \eqref{sumpi}, then $0<\pi_J<1$ for all $J\in[n]^{m-1}$. In this case, to determine whether an interval tensor $\mathcal{A}^I$ with positive row sums is an interval $B_{\pi}^{R^I}$-tensor, it is necessary to verify whether at least $n^{m-1}$ tensors are $B_{\pi}^{R}$-tensors, i.e., all tensors $\mathcal{A}^{(J)}$ as defined in Corollary \ref{c41}. In other words, none of these tensors $\mathcal{A}^{(J)}$ can be omitted.

\begin{example}\label{e4.2}
Given a vector $\pi=(\pi_1,\pi_2,\cdots,\pi_n)>0$ satisfy condition \eqref{sumpi}. Fixed $i_0\in[n]$ and $J_0\in[n]^{m-1}\setminus{(i_0,\cdots,i_0)}$. Consider the interval tensor $\mathcal{A}^I=[\underline{\mathcal{A}},\overline{\mathcal{A}}]$, where $\underline{\mathcal{A}}=(\underline{a}_{i_1i_2\cdots i_m})$ is given by
$$\underline{a}_{i_1i_2\cdots i_m}=\begin{cases}
1,\mbox{ if }i_1=i_2=\cdots =i_m,\\
0,\mbox{ otherwise},
\end{cases}$$
and $\overline{\mathcal{A}}=(\overline{a}_{i_1i_2\cdots i_m})$ is given by
$$\overline{a}_{i_1i_2\cdots i_m}=\begin{cases}
1,\mbox{ if }i_1=i_2=\cdots =i_m,\\
\frac{\pi_{J_0}}{1-\pi_{J_0}}, \mbox{ if } i_1=i_0\mbox{ and }(i_2,\cdots,i_m)=J_0,\\
0,\mbox{ otherwise}.
\end{cases}$$
For every $J\in[n]^{m-1}\setminus J_0$, the associated tensor $\mathcal{A}^{(J)}$ equals the identity tensor $\mathcal{I}$, which is trivially a $B_{\pi}^{R}$-tensor, hence each of the $n^{m-1}-1$ tensors $\mathcal{A}^{(J)}$ is a $B_{\pi}^{R}$-tensor. Yet this cannot ensure that $\mathcal{A}^{(J_0)}$ is also a $B_{\pi}^{R}$-tensor, that is this fails to guarantee that all $\mathcal{A}\in \mathcal{A}^I$ are $B_{\pi}^{R}$-tensors. In fact,
$$\pi_{J_0}R_{i_0}(\mathcal{A}^{(J_0)})=\pi_{J_0}(1+\frac{\pi_{J_0}}{1-\pi_{J_0}})=\frac{\pi_{J_0}}{1-\pi_{J_0}}=a_{i_0J_0}^{(J_0)},$$
which does not satisfy the definition of a $B_{\pi}^{R}$-tensors. Hence $\mathcal{A}^I$ is not an interval $B_{\pi}^{R^I}$-tensor.
\end{example}
This example illustrates that even when $n^{m-1}-1$ out of the $n^{m-1}$ tensors $\mathcal{A}^{(J)}$ in Corollary \ref{c41} are $B_{\pi}^{R}$-tensors, one cannot conclude that $\mathcal{A}^I$ is an interval $B_{\pi}^{R^I}$-tensor, all $n^{m-1}$ must be checked.

For an interval tensor $\mathcal{A}^I$ in $T_{m,n}$ and for any $J\in[n]^{m-1}$, define
$$\overline{M}_{J}=\max\limits_{\substack{i\in[n]\\ (i,\cdots,i)\neq J}}\overline{M}_{iJ}=\max\{ \frac{\overline{a}_{iJ}}{\overline{a}_{iJ}+\sum\limits_{\substack{K\in[n]^{m-1}\\ K\neq J}}\underline{a}_{iK}},\frac{\overline{a}_{iJ}}{R_i(\overline{\mathcal{A}})}\}.$$

Furthermore, using the corrected characterization of deterministic $B_{\pi}^{R}$-tensors from Proposition \ref{p2}, we can translate the above criterion into a unified inequality condition on the vector $\pi$.

\begin{corollary}\label{c42}
Let $\mathcal{A}^I$ be an interval tensor in $T_{m,n}$ with $R_i(\underline{\mathcal{A}})>0$ for all $i\in[n]$. Then $\mathcal{A}^I$ is an interval $B_{\pi}^{R^I}$-tensor if and only if the nonnegative vector $\pi$ satisfies condition \eqref{sumpi}, and for all $J\in[n]^{m-1}$,
$\pi_J>\overline{M}_J.$
\end{corollary}

Similar to the deterministic case, pre-specifying an appropriate vector $\pi$ for interval tensors can be challenging in practice. To address this, we establish a constructive sufficient condition for the existence of such a vector $\pi$ in the interval setting.
\begin{theorem}\label{thpi}
Let $\mathcal{A}^I$ be an interval tensor in $T_{m,n}$, and $R_i(\underline{\mathcal{A}})>0$ for all $i\in[n]$. If $$\overline{S}:=\sum\limits_{i=1}^n[\max\limits_{\{J\in[n]^{m-1}:i\in J\}}\overline{M}_J]_+^{\frac{1}{m-1}}<1,$$
then there exists a constructible positive vector $\pi\in\mathbb{R}^n$ satisfying condition \eqref{sumpi} such that $\mathcal{A}^I$ is an interval $B_{\pi}^{R^I}$-tensor.
\end{theorem}
\begin{proof}
For any $\mathcal{A}\in \mathcal{A}^I$ and $i\in[n]$, we have $R_i(\mathcal{A})\geq R_i(\underline{\mathcal{A}})>0$. Moreover, for $J\in[n]^{m-1}\setminus\{(i,\cdots,i)\}$, the definitions of $\overline{M}_{iJ}$ and $\overline{M}_J$ yield
$$\frac{a_{iJ}}{R_i(\mathcal{A})}\leq \overline{M}_{iJ}\leq\overline{M}_J.$$
For each $i\in[n]$, choose any $0<\xi\leq\frac{1-\overline{S}}{n}$, and set
$$\pi_i=[\max\limits_{\{J\in[n]^{m-1}:i\in J\}}\overline{M}_J]_+^{\frac{1}{m-1}}+\xi>0$$
(For instance, one may simply take $\xi=\frac{1-\overline{S}}{n}$). Then
$$\sum_{i=1}^n \pi_i=\overline{S}+n\cdot\frac{1-\overline{S}}{n}=1,$$
and as before,
$$\sum\limits_{J\in[n]^{m-1}} \pi_{J}=\sum_{i_2,\cdots,i_m=1}^n \pi_{i_2}\cdots \pi_{i_m}=(\sum_{i=1}^n \pi_i)^{m-1}=1,$$
so $\pi$ satisfies condition \eqref{sumpi}.

Now for any $i\in[n]$ and any $J=(i_2,\cdots,i_m)\in[n]^{m-1}\setminus\{(i,\cdots,i)\}$, let
$$\pi_{J}=\pi_{i_2}\cdots\pi_{i_m}>\prod\limits_{k=2}^m[\max\limits_{\{K\in[n]^{m-1}:i_k\in K\}}\overline{M}_K]_+^{\frac{1}{m-1}}\geq \frac{a_{iJ}}{R_i(\mathcal{A})},$$
which gives
$$\pi_{J}R_i(\mathcal{A})>a_{iJ}.$$
Therefore, there exists a constructible positive vector $\pi\in\mathbb{R}^n$ satisfying condition \eqref{sumpi}, such that $\mathcal{A}$ is a $B_{\pi}^{R}$-tensor.

\end{proof}

\begin{corollary}
Let $\mathcal{A}^I$ be an interval tensor in $T_{m,n}$ with $\underline{A}\geq0$ and $R_i(\underline{A})>0$ for all $i\in[n]$. If $$\sum\limits_{i=1}^n(\max\limits_{\{J\in[n]^{m-1}:i\in J\}}\overline{M}_J)^{\frac{1}{m-1}}<1,$$
then there exists a constructible positive vector $\pi\in\mathbb{R}^n$ satisfying condition \eqref{sumpi}, such that $\mathcal{A}^I$ is an interval $B_{\pi}^{R^I}$-tensor.
\end{corollary}
\begin{proof}
Since $\underline{A}\geq0$. all $\overline{M}_J$ are nonnegative, so $[\cdot]_+$ can be dropped. The statement the follows directly from Theorem \ref{thpi}.
\end{proof}
\begin{example}
Let $\mathcal{A}^I=[\underline{\mathcal{A}},\overline{\mathcal{A}}]$, where $\underline{a}_{111}=\underline{a}_{222}=\overline{a}_{111}=\overline{a}_{222}=3$, $\underline{a}_{122}=\underline{a}_{211}=\overline{a}_{122}=\overline{a}_{211}=0.2$, $\underline{a}_{112}=\underline{a}_{121}=\underline{a}_{221}=\underline{a}_{212}=0$ and $\overline{a}_{112}=\overline{a}_{121}=\overline{a}_{221}=\overline{a}_{212}=0.6$.
We have $R_1(\underline{\mathcal{A}})=R_2(\underline{\mathcal{A}})=3.2>0,$ $\overline{M}_{12}=\frac{3}{19}$, $\overline{M}_{21}=\frac{3}{19}$, $\overline{M}_{11}=\frac{1}{16}$, and $\overline{M}_{22}=\frac{1}{16}$, then
$$\overline{S}=(\max\limits_{J:1\in J}\overline{M}_J)^{\frac{1}{2}}+(\max\limits_{J:2\in J}\overline{M}_J)^{\frac{1}{2}}=2\times\sqrt{\frac{3}{19}}<1.$$
According to Theorem \ref{thpi}, we can take $\xi=0.1$, then $\pi=(\pi_1,\pi_2)=(\frac{10\sqrt{57}+19}{190},\frac{10\sqrt{57}+19}{190})$ satisfies condition \eqref{sumpi}, and $\mathcal{A}^I$ is an interval $B_{\pi}^{R^I}$-tensor.
\end{example}

Similar to Corollary \ref{c3.2}, we obtain a sufficient condition of positive definiteness for a symmetric interval tensor, which is equivalent to all tensors in the symmetric interval tensor to have positive H-eigenvalues.
\begin{corollary}\label{c4.5}
Let $m$ be even, and let $\mathcal{A}^I$ be a symmetric interval tensor in $T_{m,n}$ with $R_i(\underline{A})>0$ for all $i\in[n]$. If $\overline{S}<1$,
then $\mathcal{A}^I$ is a positive definite interval tensor.
\end{corollary}

For the important subclass of interval $Z$-tensors, the interval $B_{\pi}^{R^{I}}$ property is equivalent to strict diagonal dominance, which extends the corresponding result for deterministic tensors.

\begin{proposition}\label{p4.3}
Let $\mathcal{A}^I$ be an interval $Z$-tensor in $T_{m,n}$, and $\pi\in\mathbb{R}^n$ be a positive vector satisfying condition $\eqref{sumpi}$. Then $\mathcal{A}^I$ is an interval $B_{\pi}^{R^I}$-tensor if and only if $\underline{\mathcal{A}}$ is strictly diagonally dominant.
\end{proposition}
\begin{proof}
``{\bf Necessity.}'' Since $\underline{\mathcal{A}}\in \mathcal{A}^I$, then $\underline{\mathcal{A}}$ is a $B_{\pi}^{R}$-tensor and also a $Z$-tensor. According to Proposition \ref{p4}, $\underline{\mathcal{A}}$ is strictly diagonally dominant.

``{\bf Sufficiency.}'' Let $\underline{\mathcal{A}}$ be strictly diagonally dominant, then for all $i\in[n]$,
$$\underline{a}_{ii\cdots i}>\sum\limits_{J\in[n]^{m-1}\setminus\{(i,\cdots,i)\}}|\underline{a}_{iJ}|.$$
As $\underline{\mathcal{A}}\in \mathcal{A}^I$ and  $\mathcal{A}^I$ is an interval $Z$-tensor, for every $\mathcal{A}\in \mathcal{A}^I$ and all $i\in[n]$,
$$a_{ii\cdots i}\geq\underline{a}_{ii\cdots i}>\sum\limits_{J\in[n]^{m-1}\setminus\{(i,\cdots,i)\}}|\underline{a}_{iJ}|\geq \sum\limits_{J\in[n]^{m-1}\setminus\{(i,\cdots,i)\}}|a_{iJ}|.$$
Hence, each $\mathcal{A}\in \mathcal{A}^I$ is strictly diagonally dominant. Therefore, by Proposition \ref{p4}, each $\mathcal{A}\in \mathcal{A}^I$ is a $B_{\pi}^R$-tensor, then $\mathcal{A}^I$ is an interval $B_{\pi}^{R^I}$-tensor.
\end{proof}

From Proposition \ref{thp}, we have the following conclusion.
\begin{corollary}
Let nonnegative vector $\pi\in\mathbb{R}^n$ satisfy condition \eqref{sumpi}, $\sigma$ be a permutation of $[m]$ satisfying $\sigma(1) = 1$. If $\mathcal{A}^I = [\underline{\mathcal{A}}, \overline{\mathcal{A}}]$ is an interval $B_{\pi}^{R^I}$-tensor in $T_{m,n}$, then $(\mathcal{A}^I)^\sigma=[\underline{\mathcal{A}}^\sigma, \overline{\mathcal{A}}^\sigma]$ is also an interval $B_{\pi}^{R^I}$-tensor.
\end{corollary}

\section{Applications}
\setcounter{section}{5}
The concept of interval $B_{\pi}^{R^I}$-tensors introduced in this paper provides practical sufficient conditions for identifying interval $P$-tensors, and in the even-order symmetric case, positive definite interval tensors. The verification of the interval $B_{\pi}^{R^I}$-property can be carried out through two approaches, i.e., Theorem \ref{th1} or Theorem \ref{thpi}.

In polynomial optimization, the positive definiteness of an even-order symmetric tensor is equivalent to the positivity of the associated homogeneous polynomial. When the polynomial coefficients are subject to uncertainty or lie within prescribed intervals, the problem reduces to determining whether every tensor in the interval is positive definite. Corollary \ref{c4.5} together with Theorems \ref{th1} and \ref{thpi} provides verifiable sufficient conditions that require only elementary algebraic operations on the endpoint tensors, without resorting to semi-definite programming or sum-of-squares representations.

Tensor complementarity problems (TCPs), high-dimensional generalizations of linear complementarity problems, seek a vector $x\in\mathbb{R}^n$ such that $$\mathbf{x} \geq 0, \quad \mathcal{A}\mathbf{x}^{m-1} + \mathbf{q} \geq 0, \quad \mathbf{x}^{\top} (\mathcal{A}\mathbf{x}^{m-1} + \mathbf{q}) = 0,$$
Widely applied in game theory, contact mechanics, and equilibrium models, their well-posedness depends on tensor structural properties \cite{sq2015}. For deterministic $B_{\pi}^{R}$-tensors with $\pi\geq0$, corresponding TCPs have nonempty and compact solution sets for any $\mathbf{q} \in \mathbb{R}^n$ as they are $P$-tensors \cite{bhw2016}. Our interval $B_{\pi}^{R^I}$-tensor theory extends this to interval TCPs (ITCPs): if $\mathcal{A}^I$ is an interval $B_{\pi}^{R^I}$-tensor with $\pi\geq0$, the ITCP has a nonempty and compact solution set for all $\mathcal{A}\in\mathcal{A}^I$ and $q\in\mathbb{R}^n$.

Classical sufficient conditions for positive definiteness of symmetric tensors include strict diagonal dominance. For $Z$-tensors, Proposition \ref{p4} states that when the vector is positive, the $B_{\pi}^R$-property is equivalent to strict diagonal dominance. Proposition \ref{p4.3} extends this equivalence to the interval setting. However, when the interval tensor is not a $Z$-tensor (i.e., some off-diagonal upper bounds are positive), this equivalence breaks down and classical criteria become inapplicable. The following examples demonstrate that Theorems \ref{th1} and \ref{thpi} fill precisely this gap, providing effective tools for positive definiteness verification when the interval tensor is not of $Z$-type or strict diagonal dominance. While Examples \ref{e4.1} and \ref{e4.2} in Section 4 have already illustrated the direct verification process for Theorems \ref{th1} and \ref{thpi}, we supplement them here with two more general cases.

\begin{example}
Let $\mathcal{A}^I = [\underline{A}, \overline{A}]$ be a symmetric interval tensor in $T_{4,2}$. Take vector $\pi = (\frac{14}{29}, \frac{14}{29})$, which clearly satisfies condition \eqref{sumpi}, and for every $J \in [2]^3$, $\pi_J = \frac{2744}{24389}$. The lower and upper bounds are defined as follows:

$$\begin{aligned}
\underline{a}_{1111} &= 6, \quad \underline{a}_{2222} = 6, \qquad
\overline{a}_{1111} = 7, \quad \overline{a}_{2222} = 7; \\
\underline{a}_{iJ} &= 1 \quad \text{for all } J\in[2]^3\setminus\{(i,i,i)\},\\
\overline{a}_{iJ} &= \frac{3}{2} \quad \text{for all } J\in[2]^3\setminus\{(i,i,i)\}.
\end{aligned}$$
For the first row, the lower bound diagonal entry is $\underline{a}_{1111} = 6$, while the sum of absolute values of the seven off-diagonal entries is $7 \times 1 = 7>6=\underline{a}_{1111}$, strict diagonal dominance does not hold. The same holds for the second row by symmetry. And Proposition \ref{p4.3} does not apply.

We now apply Theorem \ref{th1}. The lower bound row sum is $R_1(\underline{A}) = 6 + 7 \times 1 = 13 > 0$, condition $(a)$ holds. Since $0 < \pi_J \le 1$, we use sub-condition $(b_1)$, for all $i \in [2]$ and $J \neq (i,i,i)$,

$$\underline{a}_{iiii} + \sum_{\substack{K \neq J \\ K \neq (i,i,i)}} \underline{a}_{iK} > \Bigl(\frac{1}{\pi_J}-1\Bigr)\overline{a}_{iJ}.$$

Substituting $\frac{1}{\pi_J}-1 = \frac{21645}{2744}$, the left-hand side equals $R_i(\underline{A}) - \underline{a}_{iJ} = 13 - 1 = 12$, and the right-hand side equals $\frac{21645}{2744} \times \frac{3}{2} = \frac{64935}{5488} <12$, and the inequality holds strictly for all $i$, $J$.

All conditions of Theorem \ref{th1} are satisfied, hence $\mathcal{A}^I$ is an interval $B_{\pi}^{R^I}$-tensor. Since it is an even-order symmetric interval tensor, Corollary \ref{c4.5} guarantees that $\mathcal{A}^I$ is a positive definite interval tensor
\end{example}
When a suitable vector $\pi$ is not easy to guess, the constructive approach of Theorem \ref{thpi} can be used.

\begin{example}
 Let $\mathcal{A}^I$ be a symmetric interval tensor in $T_{4,2}$. The lower bounds of all off-diagonal entries are $\frac{1}{20}$, with diagonal entries equal to $1$, and the upper bounds are defined as follows:
$$\overline{a}_{1111} = \overline{a}_{2222} = \frac{3}{2},$$
$$\overline{a}_{1112} = \overline{a}_{1121} = \overline{a}_{1211} = \overline{a}_{2221} = \overline{a}_{2212} = \overline{a}_{2122} = \frac{37}{200},$$
$$\overline{a}_{1122} = \overline{a}_{1212} = \overline{a}_{1221} = \overline{a}_{2211} = \overline{a}_{2121} = \overline{a}_{2112} = \frac{29}{200} ,$$
$$\overline{a}_{1222} = \overline{a}_{2111} = \frac{1}{8}.$$
This tensor also satisfies the premise of Theorem \ref{thpi}, with row sums of the lower bound tensor $$R_1(\underline{\mathcal{A}})=R_2(\underline{\mathcal{A}})=\frac{27}{20}>0.$$
It is not an interval $Z$-tensor because $\overline{a}_{1112}=\frac{37}{200}>0$, nor is it a strictly diagonally dominant interval tensor because $\underline{a}_{1111}=1<\frac{223}{200}$.

Applying Theorem \ref{thpi} for verification, the row sums of the upper bound tensor are $\frac{523}{200}$. Calculations yield $$\overline{M}_{111}=\overline{M}_{222}=\frac{5}{57},\ \  \overline{M}_{112}=\frac{37}{297},\mbox{ and } \overline{M}_{122}=\frac{29}{289},$$
leading to
$$\overline{S}=2\times\left(\frac{37}{297}\right)^{1/3}<1,$$
which satisfies the core condition. Constructing the positive vector $\pi_1=\pi_2=\frac{1}{2}$, verification shows $\pi_{111}=\pi_{222}=\frac{1}{8}>\frac{5}{57}$, $\pi_{112}=\frac{1}{8}>\frac{37}{297}$, and $\pi_{122}=\frac{1}{8}>\frac{29}{289}$, so $\mathcal{A}^I$ is an interval $B_{\pi}^{R^I}$-tensor. Since $m=4$ is even, $\mathcal{A}^I$ is symmetric, and $\pi\geq0$, by Corollaries \ref{c4.1} and \ref{c4.5}, this interval tensor is both an interval $P$-tensor and a positive definite interval tensor.
\end{example}

\section{Conclusions}
In this paper, we have systematically investigated the theory and recognition criteria for interval $B_{\pi}^{R^I}$-tensors. We provide a corrected characterization of $B_{\pi}^{R^I}$-tensors, several new properties of individual $B_{\pi}^{R^I}$-tensors, and the discrimination theory for interval $B_{\pi}^{R^I}$-tensors, featuring algebraic equivalent conditions and a constructive existence criterion. Numerical examples demonstrate that our criteria can effectively certify positive definiteness of interval tensors when classical strict diagonal dominance and $Z$-tensor criteria are inapplicable.

It is worth noting that the definition of interval $B_{\pi}^{R^I}$-tensors adopted in this paper requires the existence of a single vector $\pi$ such that every tensor in the interval is a $B_{\pi}^{R}$-tensor with respect to that same $\pi$. This approach can be viewed as a natural extension of the ``homogeneous" interval matrix definition to the tensor setting. In the matrix context, there also exists a so-called ``heterogeneous" definition, which only requires that each matrix individually admits some vector satisfying condition \eqref{sumpi}, without insisting on a unified $\pi$ for the whole interval family. For matrices, the ``homogeneous" and ``heterogeneous" definitions are equivalent; whether this equivalence extends to higher-order tensors remains an open problem. Closing the gap between the sufficient conditions established here and necessary ones is another direction for future work.

\section*{Competing interest}
The author declares that he has no known competing financial interests or personal relationships that could have appeared to influence the work reported in this paper.

\section*{Availability of data and materials}
This manuscript has no associated data or the data will not be deposited. [Authors' comment: This is a theoretical study and there are no external data associated with the manuscript].
\section*{Funding}




\begin{thebibliography}{ }
\bibitem{b1} Qi, L., Lay Teo, K.: Multivariate polynomial minimization and its application in signal processing. J. Global Optim. {\bf26}(4), 419-433 (2003).
\bibitem{b2} Song, Y., Qi, L.: Boundedness from below conditions for a general scalar potential of two real scalars fields and the Higgs boson. Theor. Math. Phys. {\bf220}(3), 591-604 (2024).
\bibitem{b3} Chen, H., Qi, L., Song, Y.: Column sufficient tensors and tensor complementarity problems. Front. Math. China {\bf13}(2), 255-276 (2018).
\bibitem{b4} Bose, N.K., Kamat, P.S.: Algorithm for stability test of multidimensional filters. IEEE Trans. Acoust. Speech Signal Process. {\bf22}, 307-314 (1974).
\bibitem{ql2017} Qi, L., Luo, Z.: Tensor Analysis: Spectral Theory and Special Tensors. SIAM, Philadelphia (2017).
\bibitem{sq2015} Song, Y., Qi, L.: Properties of some classes of structured tensors. J. Optim. Theory Appl. {\bf165}(3), 854-873 (2015).
\bibitem{dlq2018} Ding, W., Luo, Z., Qi, L.: $P$-tensors, P$_0$-tensors, and their applications. Linear Algebra Appl. {\bf555}, 336-354 (2018).
\bibitem{op2021} Orera, H., Pe\~{n}a, J.M.: Error bounds for linear complementarity problems of $B_{\pi}^R$-matrices. Comput. Appl. Math. {\bf40}(3), 94:1-13 (2021).
\bibitem{qs2014} Qi, L., Song, Y.: An even order symmetric $B$ tensor is positive definite. Linear Algebra Appl. {\bf457}, 303-312 (2014).
\bibitem{cy2015} Li, C., Li, Y.: Double $B$-tensors and quasi-double $B$-tensors. Linear Algebra Appl. {\bf466}, 343-356 (2015).
\bibitem{d2} Li, C., Qi, L., Li, Y.: $MB$-tensors and $MB_0$-tensors. Linear Algebra Appl. {\bf484}, 141-153 (2015).
\bibitem{bpi2} Orera, H., Pe\~{n}a, J.M.: $B_{\pi}^R$-tensors. Linear Algebra Appl. {\bf 581}, 247-259 (2019).
\bibitem{bpi1} He, J., Liu, Y., Lv, W.: Some classes of nonsingular tensors and application. Linear Multilinear Algebra. {\bf72}(7), 1078-1093  (2024).
\bibitem{boz2020} Bozorgmanesh, H., Hajarian, M., Chronopoulos, A.Th.: Interval tensors and their application in solving multilinear systems of equations. Comput. Math. Appl. {\bf79}(3), 697-715 (2020).
\bibitem{sr2020} Rahmati, S., Tawhid, M.A.: On intervals and sets of hypermatrices (tensors). Front. Math. China {\bf15}(6), 1175-1188 (2020).
\bibitem{bf2022} Beheshti, R., Fathi, J., Zangiabadi, M.: Some classes of interval tensors and their properties. Wavelets and Linear Algebra {\bf9}(1), 49-65 (2022).
\bibitem{CZ2023} Cui, L., Zhang, X.: Bounds of H-eigenvalues of interval tensors, Comp. Appl. Math. {\bf42}, 280 (2023).
\bibitem{lm2023} Lorenc, M.: $B_{\pi}^R$-Matrices, $B$-Matrices, and Doubly $B$-Matrices in the Interval Setting. Acta Cybern. {\bf26}(1), 83-103 (2023).
\bibitem{ys2026} Ye, L., Song, Y.: Interval $B$-Tensors and Interval Double $B$-Tensors. J. Optim. Theory Appl. {\bf209}(2), 55:1-39 (2026).
\bibitem{zq2014} Zhang, L. Qi, L., Zhou, G.: M-tensors and some applications. SIAM J. on Matrix Anal. Appl. {\bf35}(2), 437-452 (2014).
\bibitem{q2005} Qi, L.: Eigenvalues of a real supersymmetric tensor. J. Symb. Comput. {\bf40}(6), 1302-1324 (2005).
\bibitem{bhw2016} Bai, X., Huang, Z., Wang, Y.: Global uniqueness and solvability for tensor complementarity problems. J. Optim. Theory Appl. {\bf170}(1), 72-84 (2016).

\end{thebibliography}

\end{document}